\documentclass[10pt]{amsart}
 
\usepackage{amsfonts}
\usepackage{amsmath} 
\usepackage{amssymb}
\usepackage{verbatim}
\usepackage[usenames]{color}
\usepackage{hyperref}
\usepackage{url}
\usepackage{mathrsfs}
\usepackage{tikz,tikz-qtree,ifthen}

\usepackage{array}
\usepackage{amscd}
\usepackage{graphicx}
\input xy
\xyoption{all}   
\usepackage{latexsym, amsthm, mathrsfs}
\usepackage[latin1]{inputenc}
\usepackage{fancybox}
\usepackage{amscd}
\usepackage[german, english]{babel}

\input{xy}
\usepackage[all]{xy}

\title[Creature Forcing and topological Ramsey spaces]{Creature Forcing and topological Ramsey spaces}

\author{Natasha Dobrinen} 
\address{University of Denver\\
Department of Mathematics, 2280 S Vine St, Denver, CO 80208, USA} 
\email{natasha.dobrinen@du.edu} 
  \urladdr{\url{http://web.cs.du.edu/~ndobrine}} 
\thanks{This research was partially done whilst the author was a visiting fellow at the Isaac Newton Institute for Mathematical Sciences in the programme `Mathematical, Foundational and Computational Aspects of the Higher Infinite' (HIF).  
Dobrinen gratefully acknowledges support from  the Isaac Newton Institute and from  National Science Foundation Grant DMS-1301665}

\subjclass[2010]{03E40, 03E02, 03E05, 05D10, 54H05}

\newcommand{\om}{\omega}

\newcommand{\vp}{\varphi}

\newcommand{\sse}{\subseteq}
\newcommand{\contains}{\supseteq}

\newcommand{\Rone}{\mathcal{R}^{\tt tt}(K_1,\Sigma^*_1)}

\DeclareMathOperator{\PC}{PC}
\DeclareMathOperator{\depth}{depth}

\DeclareMathOperator{\dom}{dom}

\DeclareMathOperator{\val}{\mathbf{val}}
\DeclareMathOperator{\dis}{\mathbf{dis}}
\DeclareMathOperator{\nor}{\mathbf{nor}}
\DeclareMathOperator{\bfH}{\mathbf{H}}
\DeclareMathOperator{\pos}{\mathrm{pos}}

\newcommand{\dn}{\mathrm{dn}}
\newcommand{\up}{\mathrm{up}}

\newcommand{\rgl}{\rangle}
\newcommand{\lgl}{\langle}

\newcommand{\re}{\restriction}
\newcommand{\lre}{\upharpoonleft}

\newcommand{\bN}{\mathbb{N}}

\newcommand{\ra}{\rightarrow}

\newtheorem{thm}{Theorem}  
\newtheorem{prop}[thm]{Proposition} 
\newtheorem{lem}[thm]{Lemma} 
\newtheorem{cor}[thm]{Corollary} 

\newtheorem{claim}[thm]{Claim}

\theoremstyle{definition}   
\newtheorem{defn}[thm]{Definition}

\theoremstyle{remark} 
\newtheorem{rem}{Remark} 
 
\newtheorem*{ack}{Acknowledgments}

\newcommand{\noprint}[1]{\relax}

\begin{document}

\maketitle

\dedicatory{Celebrating Alan Dow and his tours de force in set theory and topology}

\begin{abstract}
This article  introduces  a line  of  investigation into 
connections between creature forcings and topological Ramsey spaces.
Three examples of  sets of 
pure candidates for
creature forcings  are shown to 
contain dense subsets which are actually topological Ramsey spaces.
A new variant of the product tree Ramsey theorem is proved in order to 
obtain the 
 pigeonhole principles  for two of these examples.   
\end{abstract}


\section{Introduction}\label{sec.intro}

Connections between partition theorems and creature forcings have been known for some time.
Partition theorems are used to establish various norm functions and to deduce forcing properties, for instance, properness.
Conversely, creature forcings can give rise to new partition theorems, as  seen, for instance,   in \cite{Roslanowski/Shelah13}.
Todorcevic  pointed out to the author in 2008 that   there are strong connections between creature forcings and topological Ramsey spaces  deserving of  a systematic investigation.
The purpose of this note is to open up  this line of research and provide some tools for future investigations.

In \cite{Roslanowski/Shelah13}, Ros{\l}anowski and Shelah 
proved partition theorems for several broad classes of creature forcings.
Their partition theorems  have the following form:
Given  a  creature forcing and letting 
 $\mathcal{F}_{\mathbf{H}}$ denote the related
countable set of finitary functions,
for any  partition of $\mathcal{F}_{\mathbf{H}}$ into finitely many pieces 
there is a pure candidate for which all  finitary functions obtainable from it (the {\em possibilities} on the all creatures obtained from the pure candidate)  reside in one piece of the partition.
Their proofs proceed  in a similar vein to Glazer's 
proof of Hindman's Theorem:
Using the subcomposition function on pure candidates, they define an associative binary operation which gives rise to a semi-group on the set of creatures. 
Then they prove the existence of idempotent ultrafilters for this semi-group.
As a consequence, they obtain the partition theorems mentioned above.
In particular, assuming  the Continuum Hypothesis, there is an ultrafilter on $\mathcal{F}_{\mathbf{H}}$ which is generated by pure candidates, analogously to  ultrafilters on base set $[\om]^{<\om}$ generated by infinite block sequences using Hindman's Theorem.

In this article, we look at 
three specific examples of creature forcings from \cite{Roslanowski/Shelah13}
and  construct dense subsets of the 
collections of pure candidates  which we prove form topological Ramsey spaces; 
that is, these dense subsets satisfy the Abstract Ellentuck Theorem:  In the related exponential topology, every subset which has the property of Baire is Ramsey.
As a corollary, we  recover Ros{\l}anowski and Shelah's
partition theorems for these particular examples.

Showing that 
the Axiom \bf A.4 \rm (pigeonhole) holds for these forcings  is  quite related to, but in general not the same as, the partition theorems in \cite{Roslanowski/Shelah13}.
However, for two of these examples,  showing that there are dense subsets forming a topological Ramsey space is actually stronger, and   the related 
partition theorems in \cite{Roslanowski/Shelah13} are recovered.
For these two examples, the pigeonhole principle relies on  a Ramsey theorem
for  unbounded finite products of finite sets, where exactly one of the sets in the product can be replaced with the collection of its $k$-sized subsets.
This is proved in Theorem \ref{thm.newandimportant} in  Section \ref{sec.PTRT},
 building  on work of  Di Prisco, Llopis and Todorcevic  in   \cite{DiPrisco/Llopis/Todorcevic04}.
The method of proof for  Theorem \ref{thm.newandimportant}   lends itself to generalizations, setting the stage for future work  regarding  more types of creature forcings, as well as possible density versions of Theorem \ref{thm.newandimportant} and variants in the vein of   \cite{Todorcevic/Tyros13}, in which  Todorcevic and Tyros proved the density version of Theorem \ref{thm.3.21}.
In Section \ref{sec.tRs}, we 
show that Examples 2.10, 2.11, and 2.13 in \cite{Roslanowski/Shelah13}
have dense subsets forming topological Ramsey spaces.
Theorem \ref{thm.newandimportant} is applied to 
prove the Axiom \bf A.4 \rm for 
 Examples 2.10 and 2.11;
the Hales-Jewett Theorem is used to prove the Axiom \bf A.4 \rm for Example 2.13.

The motivation for this line of investigation is several-fold.
When a forcing has a dense set forming a topological Ramsey space, it  makes available   Ramsey-theoretic techniques
 aiding  investigations of  the properties of the generic extensions and the 
  related generic
 ultrafilter.
In particular, it makes investigations of forcing over $L(\mathbb{R})$ reasonable, as all subsets of the space in $L(\mathbb{R})$ are Ramsey.
Further, 
by work of Di Prisco, Mijares, and Nieto in \cite{DiPrisco/Mijares/Nieto15}, in the presence of a supercompact cardinal, the generic ultrafilter forced by a topological Ramsey space, partially ordered by almost reduction, has complete combinatorics in over $L(\mathbb{R})$.
Having at one's disposal the Abstract Ellentuck Theorem or the Abstract Nash-Williams Theorem 
aids in  proving canonical equivalence relations on fronts and barriers, in the vein of Pudl\'{a}k and R\"{o}dl \cite{Pudlak/Rodl82}.  
This in turn makes possible  investigations of initial Rudin-Keisler and Tukey structures below these generic ultrafilters in the line of  \cite{Raghavan/Todorcevic12},
 \cite{Dobrinen/Todorcevic14}, 
\cite{Dobrinen/Todorcevic15},
\cite{Dobrinen/Mijares/Trujillo14},
\cite{DobrinenJSL15},
and \cite{DobrinenJML16}.

For the sake of space,  we only include 
 in Section
\ref{sec.reviewtRs}
the basics of topological Ramsey spaces needed to understand the present work
and refer the reader to Todorcevic's book \cite{TodorcevicBK10} for a more thorough background. 
Likewise, we do not attempt to adequately present background material on creature forcing.
However, we include throughout this paper  references to 
Ros{\l}anowski and Shelah's book \cite{Roslanowski/ShelahBK} and their paper \cite{Roslanowski/Shelah13} 
so that the interested reader can pursue further this line of research.

\begin{ack}
The author thanks S.\ Todorcevic  for suggesting in 2008  that 
 connections between  topological Ramsey space theory  and creature forcing  is  deserving of   in-depth study.
The author extends many thanks to the referee for  thorough and detailed readings of the paper and for pertinent  comments and suggestions.  
\end{ack}

On a personal note, I  would like to thank Alan Dow for his inspiring and encouraging  influence on my early and present mathematics.  He and his work are truly exceptional.
Happy Birthday, Alan!


\section{Basics of   topological Ramsey spaces}\label{sec.reviewtRs}

A brief review of topological Ramsey spaces is provided in this section for the reader's convenience.
Building on seminal  work of Carlson and Simpson in \cite{Carlson/Simpson90}, Todorcevic distilled  key properties of the Ellentuck space into  four axioms, \bf A.1  \rm -  \bf A.4\rm, which guarantee that a space is a topological Ramsey space.
(For further background, the reader  is referred to Chapter 5 of \cite{TodorcevicBK10}.)
The  axioms \bf A.1  \rm -  \bf A.4 \rm
are defined for triples
$(\mathcal{R},\le,r)$
of objects with the following properties:
$\mathcal{R}$ is a nonempty set,
$\le$ is a quasi-ordering on $\mathcal{R}$,
 and $r:\mathcal{R}\times\om\ra\mathcal{AR}$ is a  surjective map producing the sequence $(r_n(\cdot)=r(\cdot,n))$ of  restriction maps, where
$\mathcal{AR}$ is  the collection of all finite approximations to members of $\mathcal{R}$.
For $u\in\mathcal{AR}$ and $X,Y\in\mathcal{R}$,
\begin{equation}
[u,X]=\{Y\in\mathcal{R}:Y\le X\mathrm{\ and\ }(\exists n)\ r_n(Y)=u\}.
\end{equation}

For each $n<\om$, $\mathcal{AR}_n=\{r_n(X):X\in\mathcal{R}\}$.
\vskip.1in

\begin{enumerate}
\item[\bf A.1]\rm
\begin{enumerate}
\item[(1)]
$r_0(X)=\emptyset$ for all $X\in\mathcal{R}$.\vskip.05in
\item[(2)]
$X\ne Y$ implies $r_n(X)\ne r_n(Y)$ for some $n$.\vskip.05in
\item[(3)]
$r_m(X)=r_n(Y)$ implies $m=n$ and $r_k(X)=r_k(X)$ for all $k<n$.\vskip.1in
\end{enumerate}
\end{enumerate}
According to \bf A.1 \rm (3) for each $u\in\mathcal{AR}$ there is exactly one $n$ for which there exists an $X\in\mathcal{R}$ satisfying $u=r_n(X)$.
This $n$ is called the {\em length} of $u$ and  we
write $|u|=n$.
We use the abbreviation $[n,X]$ to denote $[r_n(X),X]$.
For $u,v\in\mathcal{AR}$, we write $u\sqsubseteq v$ if and only if $(\exists X\in\mathcal{R})(\exists m\le n\in \om)(u=r_m(X)\wedge v=r_n(X))$.
We write $u\sqsubset v$ if and only if $u\sqsubseteq v$ and $u\ne v$.

\begin{enumerate}
\item[\bf A.2]\rm
There is a quasi-ordering $\le_{\mathrm{fin}}$ on $\mathcal{AR}$ such that\vskip.05in
\begin{enumerate}
\item[(1)]
$\{v\in\mathcal{AR}:v\le_{\mathrm{fin}} u\}$ is finite for all $u\in\mathcal{AR}$,\vskip.05in
\item[(2)]
$Y\le X$ iff $(\forall n)(\exists m)\ r_n(Y)\le_{\mathrm{fin}} r_m(X)$,\vskip.05in
\item[(3)]
$\forall u,v,y\in\mathcal{AR}[y\sqsubset v\wedge v\le_{\mathrm{fin}} u\ra\exists x\sqsubset u\ (y\le_{\mathrm{fin}} x)]$.\vskip.1in
\end{enumerate}
\end{enumerate}

The number $\depth_X(u)$ is the least $n$, if it exists, such that 
$u\le_{\mathrm{fin}}r_n(X)$.
If such an $n$ does not exist, then we write $\depth_X(u)=\infty$.
If $\depth_X(u)=n<\infty$, then $[\depth_X(u),X]$ denotes $[r_n(X),X]$.

\begin{enumerate}
\item[\bf A.3] \rm
\begin{enumerate}
\item[(1)]
If $\depth_X(u)<\infty$ then $[u,Y]\ne\emptyset$ for all $Y\in[\depth_X(u),X]$.\vskip.05in
\item[(2)]
$Y\le X$ and $[u,Y]\ne\emptyset$ imply that there is $Y'\in[\depth_X(u),X]$ such that $\emptyset\ne[u,Y']\sse[u,Y]$.\vskip.1in
\end{enumerate}
\end{enumerate}
Additionally, 
for $n>|u|$, let  $r_n[u,X]$  denote the collection
$\{r_n(Y):Y\in [u,X]\}$.
\begin{enumerate}
\item[\bf A.4]\rm
If $\depth_X(u)<\infty$ and if $\mathcal{O}\sse\mathcal{AR}_{|u|+1}$,
then there is $Y\in[\depth_X(u),X]$ such that
$r_{|u|+1}[u,Y]\sse\mathcal{O}$ or $r_{|u|+1}[u,Y]\sse\mathcal{O}^c$.\vskip.1in
\end{enumerate}

The  {\em Ellentuck topology} on $\mathcal{R}$ is the topology generated by the basic open sets
$[u,X]$;
it refines the  metric topology on $\mathcal{R}$,  considered as a subspace of the Tychonoff cube $\mathcal{AR}^{\bN}$.
Given the Ellentuck topology on $\mathcal{R}$,
the notions of nowhere dense, and hence of meager are defined in the usual way.
We  say that a subset $\mathcal{X}$ of $\mathcal{R}$ has the {\em property of Baire} iff $\mathcal{X}=\mathcal{O}\cap\mathcal{M}$ for some Ellentuck open set $\mathcal{O}\sse\mathcal{R}$ and Ellentuck meager set $\mathcal{M}\sse\mathcal{R}$.

\begin{defn}[\cite{TodorcevicBK10}]\label{defn.5.2}
A subset $\mathcal{X}$ of $\mathcal{R}$ is {\em Ramsey} if for every $\emptyset\ne[u,X]$,
there is a $Y\in[u,X]$ such that $[u,Y]\sse\mathcal{X}$ or $[u,Y]\cap\mathcal{X}=\emptyset$.
$\mathcal{X}\sse\mathcal{R}$ is {\em Ramsey null} if for every $\emptyset\ne [u,X]$, there is a $Y\in[u,X]$ such that $[u,Y]\cap\mathcal{X}=\emptyset$.

A triple $(\mathcal{R},\le,r)$ is a {\em topological Ramsey space} if every subset of $\mathcal{R}$  with the property of Baire  is Ramsey and if every meager subset of $\mathcal{R}$ is Ramsey null.
\end{defn}

The following result can be found as Theorem
5.4 in \cite{TodorcevicBK10}.

\begin{thm}[Abstract Ellentuck Theorem]\label{thm.AET}\rm \it
If $(\mathcal{R},\le,r)$ is closed (as a subspace of $\mathcal{AR}^{\bN}$) and satisfies axioms {\bf A.1}, {\bf A.2}, {\bf A.3}, and {\bf A.4},
then every  subset of $\mathcal{R}$ with the property of Baire is Ramsey,
and every meager subset is Ramsey null;
in other words,
the triple $(\mathcal{R},\le,r)$ forms a topological Ramsey space.
\end{thm}


\section{A variant of the  product tree Ramsey theorem}\label{sec.PTRT}

The main theorem of this section, 
Theorem  \ref{thm.newandimportant}, is a Ramsey theorem on  
unbounded finite products of finite sets.  
This is a variant of Theorem \ref{thm.3.21}
below, with the strengthenings that 
 exactly one of the entries $K_l$  in each finite product is replaced with $[K_l]^k$ and   $l$ is allowed to vary over all numbers less than or equal to the length of the product, and the weakening  that some of the chosen subsets may have cardinality one.
It seems that a full strengthening
of Theorem \ref{thm.3.21} of the form  where   the  index of the $k$-sized subsets is allowed to vary over every index $l$ may  not be possible
(see  Remark \ref{rem.theorem}).
The conclusion of Theorem  \ref{thm.newandimportant} is what is needed to 
 prove  Axiom \bf A.4 \rm for two of the examples of  forcing with pure candidates  in the next section;
 it is the essence of the pigeonhole principle for $r_k[k-1,\bar{t}\,]$, for $\bar{t}$ in a particular  dense subset of the creature forcing.
The hypothesis  in Theorem
\ref{thm.newandimportant} that the sizes of the $K_j$ grow as $j$ increases lends itself to our intended applications.

Throughout,  for $l\le n$, $[K_l]^k\times \prod_{j\in (n+1)\setminus\{l\}}K_j$ is  
used to denote 
$$
K_0\times\dots\times K_{l-1}\times [K_l]^k\times K_{l+1}\times\dots\times K_n.
$$

\begin{thm}\label{thm.newandimportant}
Given $k\ge 1$, a sequence of positive integers $(m_0,m_1,\dots)$,
sets $K_j$, $j<\om$, such that $|K_j|\ge j+1$,
and a coloring 
$$
c:\bigcup_{n<\om}\bigcup_{l\le n}([K_l]^k\times \prod_{j\in (n+1)\setminus\{l\}}K_j)\ra 2,
$$
there are infinite sets $L,N\sse \om$ such that,
enumerating $L$ and $N$ in increasing order, $l_0\le n_0<l_1\le n_1<\dots$,
 and there are subsets $H_j\sse K_j$, $j<\om$,
such that 
$|H_{l_i}|=m_i$
for each $i<\om$, 
 $|H_j|=1$ for each $j\in \om\setminus L$,
and $c$ is constant on
$$
\bigcup_{n\in N}
\bigcup_{l\in L\cap (n+1)}
([H_l]^k\times\prod_{j\in (n+1)\setminus\{l\}}H_j).
$$
\end{thm}

Theorem \ref{thm.newandimportant} is a variant of the following product tree Ramsey theorem, (Lemma 2.2 in \cite{DiPrisco/Llopis/Todorcevic04} and 
Theorem 3.21 in \cite{TodorcevicBK10}), 
which we now state since it
 will be used in the proof of Theorem \ref{thm.newandimportant}.
Let $\mathbb{N}^+$ denote the set of positive integers.

\begin{thm}[Di Prisco-Llopis-Todorcevic, \cite{DiPrisco/Llopis/Todorcevic04}]\label{thm.3.21}
There is an $R:[\mathbb{N}^+]^{<\om}\ra\mathbb{N}^+$ such that for every infinite sequence $(m_j)_{j<\om}$ of positive integers and for every coloring 
$$
c:\bigcup_{n<\om}\prod_{j\le n} R(m_0,\dots,m_j)\ra 2,
$$
there exist $H_j\sse R(m_0,\dots,m_j)$, $|H_j|=m_j$, for $j<\om$, such that $c$ is constant on the product 
$$
\prod_{j\le n} H_j
$$
for infinitely many $n<\om$.
\end{thm}

The proof of Theorem \ref{thm.newandimportant}
closely follows the line of proof of Theorem \ref{thm.3.21} as presented in \cite{TodorcevicBK10}.
It will follow from 
Corollary \ref{cor. applform} 
(proved via  Lemmas \ref{lem.3.20.Ex1.10tight} and \ref{lem.lem.3.18.Ex1.10} and 
Theorem \ref{thm.genprodtree})
along with a final application of Theorem \ref{thm.3.21}.
The following lemma and its proof are minor modifications  of
Lemma 2.1 in \cite{DiPrisco/Llopis/Todorcevic04}
(see also
  Lemma 3.20 in \cite{TodorcevicBK10}), the only difference being the use of $[H_0]^k$ in place of $H_0$.
We make the notational convention that for $n=0$, $[H_0]^k\times \prod_{j=1}^n H_j$ denotes $[H_0]^k$.

\begin{lem}\label{lem.3.20.Ex1.10tight}
For any given  $k\ge 1$ and  sequence $(m_j)_{j<\om}$ of positive integers,
there are numbers 
$S_k(m_0,\dots,m_j)$  such that for  any $n<\om$ and any coloring 
$$
c: [S_k(m_0)]^k\times\prod_{j=1}^n S_k(m_0,\dots,m_j)\ra 2,
$$
there are sets $H_j\sse S_k(m_0,\dots,m_j)$, $j\le n$, such that  
$|H_j|=m_j$ and $c$  is monochromatic on 
$[H_0]^k\times \prod_{j=1}^n H_j$.
\end{lem}

\begin{proof}
Let $S_k(m_0)$ be 
the least number $r$ such that  $r\ra(m_0)^k_2$.
This satisfies the lemma when $n=0$.
Now suppose that $n\ge 1$ and   the  numbers $S_k(m_0,\dots,m_j)$, $j<n$, have been obtained satisfying the lemma.
Let $N$ denote the number  $|[S_k(m_0)]^k|\cdot S_k(m_0,m_1) \cdots  S_k(m_0,\dots,m_{n-1})$,
and let  $S_k(m_0,\dots,m_n)=m_n\cdot 2^N$.
Given a coloring $c:[S_k(m_0)]^k\times\prod_{j=1}^n S_k(m_0,\dots,m_j)\ra 2$,
for each $t\in [S_k(m_0)]^k\times\prod_{j=1}^{n-1} S_k(m_0,\dots,m_j)$,
let $c_t$ denote the coloring on $ S_k(m_0,\dots,m_n)$ given by
$c_t(x)=c(t^{\frown}x)$, for $x\in S_k(m_0,\dots,m_n)$.
Let  $\lgl t_i:i<N\rgl$ be an 
enumeration of   the members of  $[S_k(m_0)]^k\times\prod_{j=1}^{n-1} S_k(m_0,\dots,m_j)$,
and let $K_0=S_k(m_0,\dots,m_n)$.
Given 
 $i<N$ and  $K_i$,
take $K_{i+1}\sse K_i$ of cardinality $m_n\cdot 2^{N-(i+1)}$ such that $c_{t_{i}}$ is constant on $\{{t_{i}}^{\frown}x:x\in K_{i+1}\}$.
By induction on $i<N$, we obtain $K_N\sse S_k(m_0,\dots,m_n)$ of size $m_n$ such that for each $i<N$, $c_{t_i}$ is constant on $K_{N}$.
Let $H_n=K_{N}$.
Now let $c'$ be the coloring on $[S_k(m_0)]^k\times\prod_{j=1}^{n-1}S_k(m_0,\dots,m_j)$ given by $c'(t)=c(t^{\frown} x)$, for any (every) $x\in H_n$.
By the induction hypothesis, there are $H_j\sse S_k(m_0,\dots,m_j)$ of cardinality $m_j$, $j<n$, such that $c'$ is constant on $[H_0]^k\times\prod_{j=1}^{n-1}H_j$.
Then $c$ is constant on $[H_0]^k\times\Pi_{j=1}^{n}H_j$.
\end{proof}

\begin{rem}
The case  $k=1$ is simply a re-statement of Lemma 2.1 in \cite{DiPrisco/Llopis/Todorcevic04}.  
If $m_0<k$, then the set $[S_k(m_0)]^k$ is the emptyset, so the whole product is empty and the lemma is vacuously true.
\end{rem}

\begin{rem}\label{rem.theorem}
If one wants a  generalization of  Theorem \ref{thm.3.21} where the placement of the $k$-sized subsets can range over all $l$,  it seems that only a finite version may be possible, as the bounds on the sizes of the sets needed to guarantee homogeneity depend both on $k$ and the number of products.
The proof of the following statement 
proceeds very similarly to the proof of Lemma  \ref{lem.3.20.Ex1.10tight}, with the difference that  one must consider $n$ different products instead of just one.

Given $k\ge 1$ and $n<\om$, there is a function $S_{k,n}:[\mathbb{N}^+]^{\le n}\ra \mathbb{N}^+$, depending on both $k$ and $n$, such that for each
 sequence $(m_j)_{j\le n}$ of positive integers,
for each coloring 
$$
c:\bigcup_{l\le n}  [S_{k,n}(m_0,\dots,m_l)]^k\times  
\prod_{j\in (n+1)\setminus\{l\}} S_{k,n}(m_0,\dots,m_j)\ra 2,
$$
there are subsets $H_j\sse S_{k,n}(m_0,\dots,m_j)$ such that
for each $l\le n$,
$|H_j|=m_j$ and 
$c$ is constant on 
$ [H_l]^k\times \prod_{j\in(n+1)\setminus\{l\}} H_j$.

As this theorem is not applied in this article and the proof takes up much room for notational reasons,  we merely note here the first few such numbers.
Let $r^1_k(m)$ denote the least number $r$ such that 
$r\ra(m)^k_2$, and let $r_k^{j+1}(m)$ denote the least number $r$ such that $r\ra (r_k^j(m))^k_2$.
For $n=0$,  $S_{k,0}(m_0)=r_k(m_0)$.
For $n=1$, the numbers $S_{k,1}(m_0)=2r_k(m_0)$ and 
$$
S_{k,1}(m_0,m_1)=r_k^{S_{k,1}(m_0)}(m_1)\cdot
2^{|[S_{k,1}(m_0)]^k|}.
$$
The point is that a general statement like this for infinite sequences  $(m_j)_{j<\om}$ would a priori seem the natural route to proving Theorem 
\ref{thm.newandimportant},
but as it only holds for finite sequences, we had to find a different means of proving the main theorem of this section.
\end{rem}

The following generalizes Theorem \ref{thm.3.21},  the first $R(m_0)$ being  replaced by $[R_k(m_0)]^k$, and provides 
 a step toward the proof  of Theorem \ref{thm.newandimportant}.
Its proof comes after Lemma \ref{lem.lem.3.18.Ex1.10}.

\begin{thm}\label{thm.genprodtree}
Given $k\ge 1$, there is a function $R_k:[\mathbb{N}^+]^{<\om}\ra \mathbb{N}^+$ such that for each
 sequence $(m_j)_{j<\om}$ of positive integers,
for each coloring 
$$
c:\bigcup_{n<\om}  [R_k(m_0)]^k\times  
\prod_{j=1}^n R_k(m_0,\dots,m_j)\ra 2,
$$
there are subsets $H_j\sse R_k(m_0,\dots,m_j)$ such that
$|H_j|=m_j$ and 
$c$ is constant on 
$$  [H_0]^k\times \prod_{j=1}^n H_j
$$
 for infinitely many $n$.
\end{thm}

The following Ramsey Uniformization Theorem, due to Todorcevic, appears (without proof) as Theorem 1.59 in \cite{TodorcevicBK10} and is essential to Lemma \ref{lem.lem.3.18.Ex1.10} below.
As previously no proof was available in the  literature and 
at the request of the referee,  the proof,   as communicated to the author by Todorcevic,  is included here, with notation slightly modified to cohere with this article.
In the following theorem and proof, the projective hierarchy refers to the metric topology on the Baire space, $[\om]^{\om}$.

\begin{thm}[Ramsey Uniformization Theorem, \cite{TodorcevicBK10}]
Suppose $X$ is a Polish space and $R$ is a coanalytic subset of the product $[\om]^{\om}\times X$ with the property that for all $M\in[\om]^{\om}$ there is $x\in X$ such that $R(M,x)$ holds.
Then there is an infinite subset $M$ of $\om$ and a continuous map $F:[M]^{\om}\ra X$ such that $R(N,F(N))$ holds for all $N\in[M]^{\om}$.
\end{thm}

\begin{proof}
Let $R$ be as in the hypothesis.
By the Kond\^{o} Uniformization Theorem \cite{Kondo38}, there is a coanalytic function $f\sse R$ which uniformizes $R$, meaning that $R(N,f(N))$ for each $N$ in the projection of  $R$ to $[\om]^{\om}$.
Since $X$ is Polish, there is a countable base for its topology, say $\lgl U_i:i<\om\rgl$.
Then for each $i<\om$, $f^{-1}(U_i)$ is $\Sigma^1_2$, since $f$ is $\Pi^1_1$.
Since $\Sigma^1_2$ sets  have the Ramsey property (with respect to the Ellentuck topology), the following fusion construction to obtain $M$ will complete the proof.

Since $f^{-1}(U_0)$ has the Ramsey property,
there is an $N_0\in[\om]^{\om}$ such that either 
 $[N_0]^{\om}\sse f^{-1}(U_0)$ or else
 $[N_0]^{\om}\cap f^{-1}(U_0)=\emptyset$.
Let $m_0=\min(N_0)$ and let $N_1'=N_0\setminus\{m_0\}$.
Suppose now
 $k\ge 1$ and we have chosen $N_0\contains\dots\contains N_{k-1}$,
such that letting $m_i=\min(N_i)$,  we have
 $m_0<\dots<m_{k-1}$
and for each $i<k$,
\begin{equation}\label{eq.eitheror}
 \forall s\sse\{m_j:j< i\}
 \ ( [s,N_i]\sse f^{-1}(U_i)\mathrm{\ or\ }
[s,N_i]\cap f^{-1}(U_i)=\emptyset).
\end{equation}
Let $N_k'$ denote  $N_{k-1}\setminus \{m_{k-1}\}$
and enumerate $\mathcal{P}(\{m_j:j<k\})$ as
$\lgl t_l:l<2^k\rgl$.
Since $f^{-1}(U_k)$ is Ramsey, there is an $N^0_k\in[t_0,N'_k]$ such that either
$[t_0,N^0_k]\sse f^{-1}(U_k)$ or else 
$[t_0,N^0_k]\cap f^{-1}(U_k)=\emptyset$.
For $l<2^k-1$,
having chosen $N^l_k$,
take $N^{l+1}_k\in[t_{l+1},N^l_k]$ such that 
either
$[t_{l+1},N^{l+1}_k]\sse f^{-1}(U_k)$ or else 
$[t_{l+1},N^{l+1}_k]\cap f^{-1}(U_k)=\emptyset$.
At the end of these $2^k$ many steps, take $N_k$ to be $N^{2^k-1}_k$, $m_k=\min(N_k)$,
and set $N_{k+1}'=N_k\setminus \{m_k\}$.
Note that  equation
(\ref{eq.eitheror})  now holds with $k$ substituted for $i$.

Let $M=\{m_k:k<\om\}$.
$M$ is infinite, since $m_k<m_{k+1}$, for all $k<\om$.
Let $F$ denote $f\re [M]^{\om}$.
To show that $F$ is continuous, it suffices to show that  for each $i<\om$,
$F^{-1}(U_i)$ is open in $[M]^{\om}$.
A standard and useful notation is to let $M/m_{i-1}$ denote the set of members of $M$ strictly greater than $m_{i-1}$.
Let $m_{-1}$ denote $-1$ so that $M_0/m_{-1}$ equals $M$.

\begin{claim}\label{claim.equality}
$F^{-1}(U_i)=\bigcup\{[s,M/m_{i-1}]:s\sse\{m_j:j< i\}$ and $[s,N_i]\sse f^{-1}(U_i)\}$.
\end{claim}

\begin{proof}
First note that $\{[s,M/m_{i-1}]:s\sse\{m_j:j< i\}\}$ partitions $[M]^{\om}$ into a  disjoint union of finitely many clopen sets, in the subspace topology on $[M]^{\om}$ inherited from $[\om]^{\om}$.
For each $s\sse\{m_j:j< i\}$, by  equation
(\ref{eq.eitheror}),
one of two cases holds:
If $[s,N_i]\sse f^{-1}(U_i)$,
then 
\begin{equation}
[s,M/m_{i-1}]= [s,N_i]\cap [M]^{\om}\sse f^{-1}(U_i)\cap [M]^{\om}=F^{-1}(U_i).
\end{equation}
If $[s,N_i]\cap f^{-1}(U_i)=\emptyset$,
then since $[s,M/m_{i-1}]\sse [s,N_i]$ and $F^{-1}(U_i)\sse f^{-1}(U_i)$,
it follows that $[s,M/m_{i-1}]\cap F^{-1}(U_i)=\emptyset$.
Thus, the Claim holds.
\end{proof}

Therefore, $F^{-1}(U_i)$ is a union of open sets in $[M]^{\om}$;
hence, $F$ is a continuous function from $[M]^{\om}$ into $X$.
\end{proof}

Given $k\ge 1$ and $M\in[\om]^{\om}$, letting $\{m_j:j<\om\}$ be the increasing enumeration of $M$, 
the notation 
 $M_o\xrightarrow{\ \ k} M_e$ means that for each 2-coloring $c:\bigcup_{n<\om}([m_1]^k\times 
\prod_{j=1}^n m_{2j+1})\ra 2$,
there are $H_j\sse m_{2j+1}$ such that $|H_j|=m_{2j}$ and $c$ is constant on $[H_0]^k\times \prod_{j=1}^n H_j$ for infinitely many $n$.
The following lemma  and its proof are almost identical with those of Lemma 3.18 in \cite{TodorcevicBK10}, the only changes being the substitution of
$\bigcup_{n<\om}( [\om]^k\times \om^{n-1})$
for the domain of the function  $c$  in place of $\om^{<\om}$, 
 the substitution of 
 $[\om]^k$ for one of the copies of $\om$, and  an application of 
Lemma \ref{lem.3.20.Ex1.10tight}
in place of the application of Lemma  3.20  in \cite{TodorcevicBK10}.
Thus, we omit its proof.

\begin{lem}\label{lem.lem.3.18.Ex1.10}
For each $k\ge 1$, there is an infinite subset $N\sse\om$ such that $M_o \xrightarrow{\ \ k} M_e$ for each $M\in[N]^{\om}$.
\end{lem}

The next proof proceeds by slight modification to the proof of  Theorem \ref{thm.3.21}, replacing $R(m_0)$ there with $[R_k(m_0)]^k$ and  replacing an instance of 
Lemma 3.18 in \cite{TodorcevicBK10} with Lemma \ref{lem.lem.3.18.Ex1.10}.

\vskip.1in

\noindent \it Proof of Theorem \ref{thm.genprodtree}. \rm
Pick an infinite subset $N=(n_p)_{p<\om}$ of positive integers enumerated in increasing order and satisfying  Lemma \ref{lem.lem.3.18.Ex1.10}.
For each $j<\om$, set
$$
R_k(m_0,\dots,m_j)=n_{2(\sum_{i=0}^j m_i)+1}.
$$
Then for every infinite sequence $(m_j)_{j<\om}$ of positive integers, 
if we let 
$$
P=\{
n_{2(\sum_{i=0}^j m_i)+\varepsilon}:j\in\om,\ \varepsilon<2\},
$$
then $P$ is an  infinite subset  of $N$ satisfying 
$P_o=(R_k(m_0,\dots,m_j))_{j<\om}$,
while the sequence $P_e$ pointwise dominates our given sequence $(m_j)_{j<\om}$.
By our choice of $N$, it follows that $P_o\xrightarrow{\ \ k}  P_e$.
$P_o$ supplies the infinitely many levels of $n$ satisfying the theorem.
\hfill $\square$

 The following corollary forms the basis of  the proof of Theorem \ref{thm.newandimportant} below.

\begin{cor}\label{cor. applform}
Let $L,N$ be  infinite subsets of $\om$ such that $l_0\le n_0<l_1\le n_1<\dots$.
Let  $k\ge 1$, $m_0\ge 1$,  and 
$K_j$, $j\ge l_0$, be nonempty sets 
 with $|K_{l_0}|=R_k(m_0)$,
 $|K_{l_i}|\ge i$ for each $i\ge 1$, 
and  $|K_j|=1$ for each $j\in(l_0,\om)\setminus L$.
Then for each coloring
$$
c:\bigcup_{n\in N} ([K_{l_0}]^k\times\prod_{j\in (l_0,n]} K_j)\ra 2,
$$
and each $r<\om$,
there are  infinite $L'\sse L$, $N'\sse N$  with $l'_0=l_0\le n'_0<l'_1\le n'_1<\dots$, and there are $H_j\sse K_j$ such that 
$|H_{l_0}|=m_0$,
$|H_{l'_i}|=r+i$ for each $i\ge 1$,
$|H_j|=1$ for each $j\in(l_0,\om)\setminus L'$,
and $c$ is constant on 
$$
\bigcup_{n\in N'} ([H_{l_0}]^k\times\prod_{j\in (l_0,n]} H_j).
$$
\end{cor}

\begin{proof}
Let $r<\om$ be fixed.
Take $(i_p)_{p<\om}$ a strictly increasing sequence   so that
$i_0=0$ and   $|K_{l_{i_p}}|\ge R_k(m_0,r+1\dots,r+p)$.
For each $j\in (l_0,\om)\setminus \{l_{i_p}:p\ge 1\}$,
take $H_j\sse K_j$ of size one.
Then the coloring $c$ on
$$
\bigcup_{n\in N}[K_{l_0}]^k\times\prod\{K_{l_{i_p}}: p\ge 1\mathrm{\ and\ } l_{i_p}\le n\}
\times
\prod\{H_j:j\in (l_0,n]\setminus\{l_{i_q}:q\ge 1\}\}
$$
induces a coloring $c'$ on 
$\bigcup_{p<\om}[J_0]^k\times\prod_{q=1}^p J_q$,
where 
$J_q=K_{l_{i_q}}$, 
as follows:
For $p<\om$ and $(X_0,x_1,\dots, x_p)\in [J_0]^k\times\prod_{q=1}^p J_q$,
letting $Y_{l_0}=X_0$, $y_{l_{i_q}}=x_q$, and for each $j\in(l_0,n_{i_p}]\setminus \{l_{i_q}:q\le p\}$ letting $y_j$ denote the member of $H_j$, 
we define
$c'(X_0,x_1,\dots,x_p)=c(Y_{l_0},y_{l_0+1},\dots,y_{n_{i_p}})$.
Apply  Theorem \ref{thm.genprodtree}   to   $c'$ 
to obtain $H_{l_0}\in[K_{l_0}]^{m_0}$, subsets $H_{l_{i_p}}\in [K_{l_{i_p}}]^{r+p}$ for each $p\ge 1$,
and an infinite set $P$ such that $c'$ is constant on 
$\bigcup_{p\in P} [H_{l_0}]^k\times\prod_{1\le q\le p} H_{l_{i_q}}$.
Then letting $N'= \{n_{i_p}:p\in P\}$, $c$ is constant 
on 
$\bigcup_{n\in N'}[H_{l_0}]^k\times\prod_{l_0<j\le n} H_n$.
Letting $L'=\{l_{i_p}:p\in P\}$ finishes the proof.
\end{proof}

Now we are equipped to prove Theorem \ref{thm.newandimportant}.
\vskip.1in

\noindent \it Proof of Theorem \ref{thm.newandimportant}. \rm
Take $l_0$ least such that $|K_{l_0}|\ge R_k(R(m_0))$,
and let $L_0=N_0=[l_0,\om)$.
For each $j<l_0$, take some $H_j\in[K_j]^1$ and let $h\re l_0$ denote $\prod_{j<l_0}H_j$.
Then $c$ restricted to $\bigcup_{n\in N_0} (h\re l_0)\times [K_{l_0}]^k\times\prod_{j\in (l_0,n]}K_j$
induces a $2$-coloring on   $\bigcup_{n\in N_0}  [K_{l_0}]^k\times\prod_{j\in (l_0,n]}K_j$.
By Corollary \ref{cor. applform}, there are infinite $L_0'\sse L_0$ and  $N_0'\sse N_0$  such that  $l_0=l'_0\le n'_0<l'_1\le n'_1<\dots$,
and there are subsets $H^0_j\sse K_j$, $j\ge l_0$, such that 
$|H^0_{l_0}|=R(m_0)$,
$|H^0_{l'_i}|=i$ for each $i\ge 1$,
$|H^0_j|=1$ for each $j\in \om\setminus L'_0$,
and $c$ is constant on $\bigcup_{n\in N'_0}(h\re l_0)\times[H^0_{l_0}]^k\times\prod_{j\in(l_0,n]}H^0_j$.

Let $H_{l_0}=H^0_{l_0}$, and 
let  $n_0=\min(N'_0)$.
Then $n_0\ge l_0$.
Let $R^1_k(m)$ denote $R_k(m)$ and in general, let $R^{i+1}_k(m)$ denote $R_k(R_k^i(m))$.
Fix an
$l_1\in L'_0$ such that  $l_1>n_0$  and
$|H^0_{l_1}|\ge R^{R(m_0)}_k(R(m_0,m_1))$.
For $j\in(l_0,l_1)$, fix some $H_j\in [H_j^0]^1$, and 
let $h\re l_1$ denote $\prod_{j\in l_1\setminus\{l_0\}} H_j$.
Enumerate $H_{l_0}$ as $\{h_{l_0}^i:i<m_0\}$.
Successively apply 
 Corollary \ref{cor. applform}
$R(m_0)$ times 
to obtain $L_1\sse L_0'$ with $\min(L_1)=l_1$, $N_1\sse N_0'$,  $H_{l_1}\sse K_{l_1}$ of cardinality $R(m_0,m_1)$,
and subsets $H^1_j\sse K_j$  for $j\in [l_1,\om)$, such that 
  listing $L_1$ as $l_1=l^1_1<l^1_2<\dots$ we have  $|H^1_{l_i}|\ge i$ and satisfying the following:
For each fixed $h^*_{l_0}\in H_{l_0}$,
the coloring 
$c$ is constant on 
$$
\bigcup_{n\in N_1} (h\re l_1)\times \{h^*_{l_0}\}\times [H_{l_1}]^k\times \prod_{j\in (l_1,n]}H^1_j.
$$

In general, suppose for  $p\ge 1$, we have fixed $l_0\le n_0<\dots<l_p\le n_p$, and chosen
 infinite sets $L_p$, $N_p$ with $l_p=\min(L_p)$
and $l_p=l_p^p\le n_p^p<l_{p+1}^{p}\le n_{p+1}^{p}<\dots$  and sets $H_j\sse K_j$ for $j\le l_p$ and sets $H_j^p\sse K_j$ for $j>l_p$ such that the following hold:
\begin{enumerate}
\item
for each $i\le p$, $|H_{l_i}|=R(m_0,\dots,m_i)$,
\item
for each 
$l\in l_p\setminus\{l_i:i<p\}$, $|H_l|=1$,
\item
for  each $i> p$, $|H^p_{l^p_i}|\ge i$,
\item
and for each $j\in (l_p,\om)\setminus L_p$, $|H^p_j|=1$.
\end{enumerate}
Let 
$h\re l_p$ denote $\prod_{j\in l_p\setminus\{l_0,\dots,l_{p-1}\}}H_j$,
which is a product of singletons.
By our construction so far, we have ensured that for each sequence $\bar{x}\in\Pi_{i<p} H_{l_i}$, $c$ is constant on 
$$
\bigcup_{n\in N_p}h\re l_p\times \bar{x}\times[H_{l_p}]^k\times\prod_{j\in (l_p,n]}H^p_j.
$$
Let $n(p)=|\prod_{i\le p} H_{l_i}|$.
Fix $n_p\in N_p$  such that $n_p\ge l_p$ and 
take $l_{p+1}\in L_p$ such that  $l_{p+1}>n_p$ and  $|H^p_{l_{p+1}}|=R^{n(p)}_k(R(m_0,\dots,m_{p+1}))$.
After  $n(p)$ successive applications of  Corollary \ref{cor. applform}, 
we obtain 
$L_{p+1}\sse L_p$ and
$N_{p+1}\sse N_p$ 
with $\min(L_{p+1})=l_{p+1}$,
subsets 
 $H_j\sse K_j$ for $j\in (l_p, l_{p+1}]$ and sets $H_j^{p+1}\sse H_j^p$ for $j> l_{p+1}$ such that the following hold:
\begin{enumerate}
\item
$|H_{l_{p+1}}|=R(m_0,\dots,m_{p+1})$,
\item
for each 
$l\in (l_p,l_{p+1})$, $|H_l|=1$,
\item
for  each $i> p+1$, $|H^{p+1}_{l^{p+1}_i}|\ge i$,
\item
and for each $j\in (l_{p+1},\om)\setminus L_{p+1}$, $|H^p_j|=1$;
\end{enumerate}
and moreover,  letting 
 $h\re l_{p+1}=\prod_{j\in l_{p+1}\setminus\{l_0,\dots,l_p\}}H_j$,
for each $\bar{x}\in  \prod_{i\le p}H_{l_i}$,
$c$ is constant on 
$$
\bigcup_{n\in N_{p+1}}(h\re l_{p+1})\times \bar{x}\times[H_{l_{p+1}}]^k\times\prod_{j\in (l_{p+1},n]}H^{p+1}_j.
$$
Then fix an $n_{p+1}\in N_{p+1}$ such that $n_{p+1}\ge l_{p+1}$.

In this manner, we obtain  $L=\{l_i:i<\om\}$ and $N=\{n_i:i<\om\}$ such that 
$
l_0\le n_0<l_1\le n_1<l_2\le n_2<\dots$,
and 
$H_j\sse K_j$, $j<\om$, such that 
 $|H_{l_i}|=R(m_0,\dots,m_i)$ for each $i<\om$,
$|H_j|=1$ for each $j\in\om\setminus L$,
and for each $p<\om$, for each $\bar{x}\in \prod_{i\le p}H_{l_i}$,
$c$ is  constant on
$$
\bigcup_{n\in N\cap[l_p,\om)} (h\re l_p)\times\bar{x}\times[H_{l_p}]^k\times\prod_{j\in (l_p,n]}H_j.
$$
Defining  $c'(\bar{x})$ to be this constant color induces a $2$-coloring on $\bigcup_{p\in\om}\prod_{i\le p}H_{l_i}$.
Since each $|H_{l_i}|=R(m_0,\dots,m_i)$,
we may apply Theorem \ref{thm.3.21} to obtain $H^*_{l_i}\sse H_{l_i}$ of cardinality $m_i$ and an infinite subset $N^*\sse N$ such that
$\bar{c}$ is constant on $\bigcup_{n\in N^*}
\prod_{i\le n}H^*_{l_i}$.
Then letting $H^*_j=H_j$  for $j\not\in L$, and letting $L^*$ be any subset of $\{l_i:i<\om\}$ such that $l^*_0\le n^*_0<l^*_1\le n^*_1<\dots$,
$c$ is constant on 
$$
 \bigcup_{n\in N^*}
\bigcup_{l\in L^*\cap (n+1)}
 [H_l^*]^k\times \prod_{j\in (n+1)\setminus \{l\}}H^*_j.
$$
\hfill$\square$


\section{Topological Ramsey spaces as dense subsets in three examples of creature forcings}\label{sec.tRs}

In  \cite{Roslanowski/Shelah13}, Ros{\l}anowski and Shelah 
proved partition theorems on countable sets of finitary functions denoted $\mathcal{F}_{\mathbf{H}}$ (see Definition \ref{defn.RS2.2} below).
Their proofs involved
using the subcomposition operation $\Sigma$ to define  a binary relation  giving rise to a
 semi-group, and then 
proving  the existence of 
 an idempotent ultrafilter (or a sequence of idempotent ultrafilters in the tight case) by utilizing the Glazer technique including applications of 
 Ellis' Lemma.
These partition theorems, stated
as  Observation 2.8 (3) and Conclusions 3.10 and 4.8  in 
\cite{Roslanowski/Shelah13},
show that, under certain assumptions on the creating pair, 
 given any finite  partition of  $\mathcal{F}_{\mathbf{H}}$, there is a pure candidate such that the collection of possibilities it codes is contained in one piece of the partition.
The terminology {\em pure candidate}  refers to  the fact  that the  infinite sequence of creatures does not have  a trunk.

Included in \cite{Roslanowski/Shelah13} are 
four specific  examples of pure candidates for creature forcings, to which these partition theorems apply.
In this section, we  show 
that  for  three of these examples, the collections of pure candidates  contain dense subsets which form topological Ramsey spaces.
We obtain as corollaries
Conclusion 4.8 for Example 2.10
(see Theorem \ref{thm.2.10RS})
and 
Observation 2.8 (3)  for Example 2.11 
(see Proposition \ref{prop.RS2.11})
in \cite{Roslanowski/Shelah13},
as the sets of possibilities from these pure candidates can be recovered from $r_1[0,\bar{t}\,]$, but not vice versa; the partition theorems for possibilities from pure candidates 
do not in general imply  the axiom \bf A.4\rm.

To show that a dense subset of a collection of pure candidates $\bar{t}$ forms a topological Ramsey space, 
it suffices by the
  Abstract Ellentuck Theorem \ref{thm.AET} to 
 define a notion of $k$-th approximation  of $\bar{t}$
and  a quasi-ordering $\le_{\rm{fin}}$ on the collection of finite approximations (in our cases this will be a partial ordering), and then prove
 that the Axioms \bf A.1 \rm - \bf A.4 \rm hold.
In each of the examples below, given a  creating pair $(K,\Sigma)$, 
we shall form a dense subset of the pure candidates, call it $\mathcal{R}(K,\Sigma)$, partially ordered by the partial ordering inherited from 
the collection of all pure candidates (see Definition 2.3 (2) in \cite{Roslanowski/Shelah13}).

For each $\bar{t}=(t_0,t_1,\dots)\in \mathcal{R}(K,\Sigma)$,
for $k<\om$, we let $r_k(\bar{t}\,)=(t_i:i<k)$.
Thus, $r_0(\bar{t}\,)$ is the empty sequence, and $r_1(\bar{t}\,)=(t_0)$,  a sequence of length one containing exactly one member of $K$.
Let $\mathcal{AR}_k$ denote $
\{r_k(\bar{t}\,):\bar{t}\in \mathcal{R}(K,\Sigma)\}$
 and 
 $\mathcal{AR}$ denote $\bigcup_{k<\om}\mathcal{AR}_k$. 
For $a\in\mathcal{AR}$ and $\bar{t}\in\mathcal{R}(K,\Sigma)$,
write $a\sqsubset \bar{t}$ if and only if $a=r_k(\bar{t}\,)$ for some $k<\om$.
The basic open sets in the Ellentuck topology are   defined as
$[a,\bar{t}\,]=\{\bar{s}\in \mathcal{R}(K,\Sigma): a\sqsubset \bar{s}$ and $\bar{s}\le \bar{t}\,\}$,
for $a\in\mathcal{AR}$ and $\bar{t}\in  \mathcal{R}(K,\Sigma)$.
Define
 the partial ordering $\le_{\mathrm{fin}}$ on $\mathcal{AR}$ 
as follows:
For $a,b\in\mathcal{AR}$, 
$b\le_{\mathrm{fin}} a$ if and only if there are $\bar{s},\bar{t}\in\mathcal{R}(K,\Sigma)$ and $j,k<\om$ such that 
$\bar{t}\le \bar{s}$,
$b=r_k(\bar{t}\,)$, $a=r_j(\bar{s}\,)$
and $m^{t_{k-1}}_{\up}=m^{s_{j-1}}_{\up}$.
Abusing  notation, we also write $a\le_{\mathrm{fin}}\bar{t}$ if for some $k<\om$, $a\le_{\mathrm{fin}}r_k(\bar{t}\,)$.
Let $\mathcal{AR}_k|\bar{t}$ denote the set of all $a\in\mathcal{AR}_k$ such that $a\le_{\mathrm{fin}}\bar{t}$,
and let $\mathcal{AR}|\bar{t}$ denote $\bigcup_{k<\om}\mathcal{AR}_k|\bar{t}$.

Given this set-up, it is clear that \bf A.1 \rm holds.
Since for each $(s_0,\dots,s_{k-1})\in\mathcal{AR}$, $\Sigma(s_0,\dots,s_{k-1})$ is  finite, \bf A.2 \rm (1) holds. 
\bf A.2 \rm (2) is simply the definition of the partial ordering $\le$ when restricted to $\mathcal{R}(K,\Sigma)$, and \bf A.2 \rm (3) is straightforward to check, using the definition of $\le_{\mathrm{fin}}$.
\bf A.3 \rm follows from the definition of $\Sigma$.
Thus, showing that the pigeonhole principle \bf A.4 \rm holds for these examples is the main focus of this section.

We now include some of the relevant creature forcing terminology.
Knowing this vocabulary is not necessary for  the proofs,
but it is included here so  the interested reader can make connections between the proofs here and the more general genre of creature forcings.
In the following three examples, FP stands for {\em forgetful partial}, which is made explicit in Context 2.1 and  Definition 2.2 in \cite{Roslanowski/Shelah13}, and is reproduced here.

\begin{defn}[\cite{Roslanowski/Shelah13}, page 356]\label{defn.RS2.2}
Let $\mathbf{H}$ be a fixed function defined on $\om$ such that $\mathbf{H}(i)$ is a finite non-empty set for each $i<\om$.
The set of all finite non-empty functions $f$ such that $\dom(f)\sse\om$ and $f(i)\in\mathbf{H}(i)$ 
(for all $i\in\dom(f))$ will be denoted by
$\mathcal{F}_{\mathbf{H}}$.

An FP creature for $\mathbf{H}$ is a tuple 
$$ t=(\nor[t],\val[t],\dis[t],m^t_{\dn},m^t_{\up})$$
such that 
\begin{enumerate}
\item[$\bullet$]
 $\nor$ is a non-negative real number, $\dis$ is an arbitrary object, and $m^t_{\dn}<m^t_{\up}<\om$, and
\item[$\bullet$] 
$\val$ is a non-empty finite subset of $\mathcal{F}_{\mathbf{H}}$ such that $\dom(f)\sse[m^t_{\dn},m^t_{\up})$ for all $f\in\val$.
\end{enumerate}
\end{defn}

Partial refers to the fact that  the domains of the functions $f\in\val$ are  allowed to be  subsets of $[m^t_{\dn},m^t_{\up})$, rather than the whole interval.
For the definition of forgetful, see Definition 1.2.5 in \cite{Roslanowski/ShelahBK}.
For the definitions of an FFCC pair, $\Sigma$, loose, tight, pure candidate, the set of possibilities pos$(\bar{t})$ and pos$^{\tt tt}(\bar{t})$ on a pure candidate $\bar{t}$, the reader is referred to Definitions 2.2 and 2.3 in \cite{Roslanowski/Shelah13}.
\vskip.1in

\bf \noindent Example 2.10 in \cite{Roslanowski/Shelah13}. \rm
Let $\mathbf{H}_1(n)=n+1$ for $n<\om$ and let $K_1$ consist of all FP creatures $t$ for $\bfH_1$ such that 
\begin{enumerate}
\item[$\bullet$]
$\dis[t]=(u,i,A)=(u^t,i^t,A^t)$, where $u\sse[m^t_{\dn},m^t_{\up})$, $i\in u$, $\emptyset\ne A\sse \mathbf{H}_1(i)$,
\item[$\bullet$]
$\nor[t]=\log_2(|A|)$,
\item[$\bullet$]
$\val[t]\sse\prod_{j\in u}\mathbf{H}_1(j)$ is such that $\{f(i):f\in\val[t]\}=A$.
\end{enumerate}
For $t_0,\dots,t_n\in K_1$ with $m^{t_l}_{\up}= m^{t_{l+1}}_{\dn}$,
let $\Sigma_1^*(t_0,\dots,t_n)$ consist of all creatures $t\in K_1$ such that
$$
m^t_{\dn}=m^{t_0}_{\dn},\ 
m^t_{\up}=m^{t_n}_{\up},\
u^t=\bigcup_{l\le n} u^{t_l},\ 
i^t=i^{t_{l^*}},\
A^t\sse A^{t_{l^*}}\mathrm{\ for\ some\ } l^*\le n,
$$
and $\val[t]\sse\{f_0\cup\dots\cup f_n:(f_0,\dots,f_n)\in\val[t_0]\times\dots \times\val[t_n]\}$.
The collection of tight  pure candidates $\PC_{\infty}^{\tt tt}(K_1,\Sigma_1^*)$  is defined in Definition 2.3 in \cite{Roslanowski/Shelah13}, pure meaning that $\bar{t}$ is an infinite sequence without a trunk.
The partial ordering $\le$ on  $\PC_{\infty}^{\tt tt}(K_1,\Sigma_1^*)$ is defined by
$\bar{t}\le \bar{s}$ if and only if there  is a
strictly increasing sequence $(j_n)_{n<\om}$ such that 
each $t_n\in\Sigma_1^*(s_{j_n},\dots,s_{j_{n+1}-1})$.

\begin{rem}
Here,
$\bar{t}$ is the stronger condition, as
this reversal of the partial order notation of Ros{\l}anowski and Shelah
is better suited to the  topological Ramsey space framework.
\end{rem}

Ros{\l}anowski and Shelah proved that 
 $(K_1,\Sigma_1^*)$ is a tight FFCC pair (see Definition 2.2 in \cite{Roslanowski/Shelah13} with bigness (Definition 2.6, \cite{Roslanowski/Shelah13}) and $t$-multiadditivity (Definition 2.5, \cite{Roslanowski/Shelah13}), and  is gluing (basically meaning that neighboring creatures can be glued together to obtain another creature - see Definition 2.1.7, \cite{Roslanowski/ShelahBK}) on every $\bar{t}\in \PC_{\infty}^{\tt tt}(K_1,\Sigma_1^*)$.
FFCC stands for smooth ([\cite{Roslanowski/ShelahBK}, 1.2.5]) {\bf F}orgetful
monotonic ([\cite{Roslanowski/ShelahBK},  5.2.3])
strongly {\bf F}initary ([\cite{Roslanowski/ShelahBK},  1.1.3, 3.3.4])
{\bf C}reature {\bf C}reating pair.


Without going into more terminology than is necessary, we point out that 
in this particular example, for $\bar{t}\in \PC_{\infty}^{\tt tt}(K_1,\Sigma_1^*)$, the set of {\em possibilities} on the pure candidate $\bar{t}$ is
\begin{equation}
\pos^{\tt tt}(\bar{t})=\bigcup\{f_0\cup\dots\cup f_n: n\in\om\wedge\forall i\le n\, (f_i\in\val[t_i])\}.
\end{equation}
For $\bar{t}\in \PC_{\infty}^{\tt tt}(K_1,\Sigma_1^*)$ and $n<\om$,
$\bar{t}\lre n$ denotes $(t_n, t_{n+1},\dots)$, the tail of $\bar{t}$ starting at $t_n$.
The following is  {\em Conclusion 4.8} in \cite{Roslanowski/Shelah13} applied to this example.

\begin{thm}[Roslanowski/Shelah, \cite{Roslanowski/Shelah13}]\label{thm.2.10RS}
Let $\bar{t}\in \PC_{\infty}^{\tt tt}(K_1,\Sigma_1^*)$,
and for each $k<\om$, let $l_k\ge 1$ and $d_k:\pos^{\tt tt}(\bar{t}\lre k)\ra l_k$ be given. 
\begin{enumerate}
\item[(a)]
There is an $\bar{s}\le \bar{t}$ in $\PC_{\infty}^{\tt tt}(K_1,\Sigma_1^*)$
such that $m^{s_0}_{\dn}=m^{t_0}_{\dn}$ and for each $i<\om$,
if $k$ is such that $s_i\in\Sigma^{\tt tt}(\bar{t}\lre k)$,
then $d_k\re\pos^{\tt tt}(\bar{s}\lre i)$ is constant.
\item[(b)]
If there is a fixed $l\ge 1$ such that for each $k<\om$, $l_k=l$,
then
there is an $\bar{s}\le \bar{t}$ in $\PC_{\infty}^{\tt tt}(K_1,\Sigma_1^*)$ and an $l'<l$ 
for each $i<\om$, if $k$ is such that $s_i\in\Sigma^{\tt tt}(\bar{t}\lre k)$ and $f\in\pos^{\tt tt}(\bar{s}\lre i)$,
then $d_k(f)=l'$.
\end{enumerate}
\end{thm}

\begin{defn}[The space $\mathcal{R}(\PC_{\infty}^{\tt tt}(K_1,\Sigma_1^*)),\le,r)$]
Let $\mathcal{R}(\PC_{\infty}^{\tt tt}(K_1,\Sigma_1^*))$
consist of those members  $\bar{t}\in
\PC_{\infty}^{\tt tt}(K_1,\Sigma_1^*)$ such that 
for each $l<\om$,
\begin{enumerate}
\item
$|A^{t_l}|=l+1$, and
\item
 for
 each $a\in A^{t_l}$,
there is exactly one function $g^{t_l}_a\in\val[t_l]$ 
such  that 
$g^{t_l}_a(i^{t_l})=a$.
\end{enumerate}
\end{defn}

It follows that for each $\bar{t}\in\mathcal{R}(\PC_{\infty}^{\tt tt}(K_1,\Sigma_1^*))$, for each $l<\om$,  
$\val[t_l]=\{g^{t_l}_a:a\in A^{t_l}\}$, and hence
$|\val[t_l]|=|A^{t_l}|=l+1$.
We point out that 
 $r_1[0,\bar{t}\,]$ consists of those $s\in \bigcup_{n<\om}\Sigma_1^*(\bar{t}\lre n)$ such that $|A^s|=|\val[s]|=1$.

\begin{thm}\label{thm.ex.1.10RamseySpace}
$\mathcal{R}(\PC_{\infty}^{\tt tt}(K_1,\Sigma_1^*)),\le,r)$ 
is a topological Ramsey space which is dense in 
the  partial ordering of all tight pure candidates  $\PC_{\infty}^{\tt tt}(K_1,\Sigma_1^*)$.
\end{thm}

\begin{proof}
Abbreviate 
$\mathcal{R}(\PC_{\infty}^{\mathrm{tt}}(K_1,\Sigma_1^*))$ as 
$\Rone$.
The space
$\Rone$
 is clearly  dense in $\PC_{\infty}^{\mathrm{tt}}(K_1,\Sigma_1^*)$.
 First we show \bf A.4 \rm holds for $r_k[k-1,\bar{t}\,]$ for all $k\ge 2$, and for $r_1[0,\bar{t}\,]\cap\Sigma^*_1(\bar{t}\,)$.
Then we will use a fusion argument to obtain \bf A.4 \rm for $r_1[0,\bar{t}\,]$.

\begin{claim}\label{claim.a}
Let $\bar{t}\in\Rone$, and 
let $C_k$ denote $r_1[0,\bar{t}\,]\cap \Sigma^*_1(\bar{t}\,)$ if $k=1$,
and  $r_k[k-1,\bar{t}\,]$ if  $k\ge 2$.
Let  $c:C_k \ra 2$ be  given.
Then
 there is an $\bar{s}\le \bar{t}$ in $\Rone$ with $m^{s_0}_{\dn}=m^{t_0}_{dn}$
such that,
if $k=1$,
$c$ is constant on $r_1[0,\bar{s}\,]\cap\Sigma^*_1(\bar{s}\,)$;
and 
if $k\ge 2$,  then $\bar{s}\in [k-1,\bar{t}\,]$ 
and  $c$ is constant on  $r_k[k-1,\bar{s}]$.
\end{claim}

\begin{proof}
Let $k\ge 1$ be fixed.
Each $\bar{x}\in C_k$ is of the form $\bar{x}=(t_0,\dots,t_{k-2},x_{k-1})$,  
where for some  $k-1\le l\le n$,
  $x_{k-1}\in\Sigma^*_1(t_{k-1},\dots,t_n)$,
$i^{x_{k-1}}=i^{t_l}$ and 
$A^{x_{k-1}}\in[A^{t_l}]^k$.
For $k=1$, $\bar{x}$ is simply $(x_0)$.
Notice that $x_{k-1}$ is completely determined by the sequence
$(n,l,A^{x_{k-1}},\lgl a_j:j\in[k-1,n]\setminus\{l\}\rgl)$, 
where  
for each $j\in[k-1,n]\setminus\{l\}$, $a_j$ is the member of $A^{t_j}$ such that every member $f\in\val[x_{k-1}]$ satisfies $f(i^{t_j})=a_j$.
Therefore,
$c$ induces a coloring on 
$$
\bigcup_{n\ge k-1}\ \left(\bigcup_{k-1\le l\le n}[A^{t_l}]^k\times\prod_{j\in [k-1,n]\setminus \{l\}} A^{t_j}\right).
$$

Apply Theorem \ref{thm.newandimportant} to the sequence of sets $A^{t_j}$, $j\ge k-1$
to obtain infinite sets $L, N$ and subsets $H_j\sse A^{t_j}$ such that 
$k-1\le l_0\le n_0<l_1\le n_1<\dots$, and 
for each  $p<\om$,  $|H_{l_p}|=k+p$, and for each $j\in\om\setminus L$, $|H_j|=1$;
and moreover, $c$ is constant on 
$$
\bigcup_{n\in N}\bigcup_{l\in L\cap (n+1)}
[H_l]^k\times\prod_{j\in   [k-1,n]\setminus \{l\}} H_j.
$$
Let $\bar{s}\in\Rone$ be defined  as follows: 
$(s_0,\dots,s_{k-2})=r_{k-1}(\bar{t})$.
For 
$p\ge k-1$, letting $q=p-(k-1)$ and $n_{-1}=k-2$,
let  $s_p$ be the member of $\Sigma^*_1(t_{n_{q-1}+1},\dots, t_{n_q})$ such that 
$i^{s_p}=i^{t_{l_q}}$,
$A^{s_p}=H_{l_q}$,
and  for all 
$f\in\val[s_p]$, $f(i^{t_j})\in H_j$ 
 for each $j\in (n_{q-1},n_q]$.
Then $\bar{s}$ is a member of $\Rone$ with  $\bar{s}\le\bar{t}$  satisfying  Claim \ref{claim.a}.
\end{proof}

Thus, we have proved \bf A.4 \rm for $r_k[k-1,\bar{t}\,]$, for  all $k\ge 2$.
As as step to proving \bf A.4 \rm for $r_1[0,\bar{t}\,]$, we first prove the following Claim.

\begin{claim}\label{claim.b}
Given $\bar{t}\in\Rone$ and colorings $c_k:
\mathcal{AR}_{k+1}|\bar{t}\ra l_k$,
 for some $l_k\ge 1$,
there is an $\bar{s}\le\bar{t}$ in $\Rone$ such that $m^{s_0}_{\dn}=m^{t_0}_{\dn}$,
$c_0$ is constant on $r_1[0,\bar{s}]\cap\Sigma^*_1(\bar{s})$, and for each $k\ge 1$,
the coloring $c_k$ on $r_{k+1}[k,\bar{s}]$ is constant.
\end{claim}

\begin{proof}
This is a standard fusion argument.
By Claim  \ref{claim.a}, there is an 
$\bar{s}^0=(s^0_0,s^0_1,\dots)$ in $\Rone$
with
$\bar{s}^0\le \bar{t}$  such that $m^{s^0_0}_{\dn}=m^{t_0}_{\dn}$ and 
$c_0$  is constant on $r_1[0,\bar{s}]\cap\Sigma^*_1(\bar{s}^0)$.
Let $s_0=s^0_0$.
Suppose $k\ge 1$ and $\bar{s}^{k-1}$ has been chosen.
Considering the coloring $c_k$ restricted to $r_{k+1}[k,\bar{s}^{k-1}]$,
 Claim  \ref{claim.a} implies  there is an $\bar{s}^k\in [k,\bar{s}^{k-1}]$ for which $c_k$ is constant on 
$r_{k+1}[k,\bar{s}^k]$.
Set $s_k=s^k_k$.
Then $\bar{s}=(s_0,s_1,\dots)\le\bar{t}$
is  in $\Rone$  and  satisfies  the Claim.
\end{proof}

Finally, to prove \bf A.4 \rm for $\mathcal{AR}_1$,
 let $c:r_1[0,\bar{t}\,]\ra 2$ be given.
For each $k<\om$,
define a coloring $c_k:\mathcal{AR}_{k+1}|\bar{t}\ra 2$ by
$c_k(x_0,\dots, x_{k})=c(\min(x_{k}))$,
where  we let $\min(x)$ denote the member of $\Sigma^*_1(x)$ such that 
$A^{\min(x)}=\{\min(A^x)\}$ and 
$\val[\min(x)]$ is the singleton $\{f\}$ where $f$ satisfies 
$f(i^{x})=\min(A^{x})$.
Take $\bar{s}\le \bar{t}$ satisfying Claim \ref{claim.b}.
Then there is 
a strictly increasing sequence $(k_j)_{j<\om}$ 
 such that
each $k_j\ge 2j+1$ and 
 the color of $c_{k_j}$ is the same on
$ r_{k_j+1}[k_j,\bar{s}\,]$,
for all $j<\om$.
Let $\bar{v}\le \bar{s}$ be the  member of $\Rone$
such that  for each $j<\om$,
$v_j\in\Sigma^*_1(s_{k_j},\dots,s_{k_{j+1}-1})$,
$i^{v_j}=i^{s_{k_{j+1}-1}}$,
and $A^{v_j}$ consists of the least $j+1$ members of $A^{s_{k_{j+1}-1}}$.
Then $\bar{v}\le\bar{t}$ and $c$ is constant on $r_1[0,\bar{v}]$.

Thus, \bf A.4 \rm holds, and therefore, by the Abstract Ellentuck Theorem \ref{thm.AET} and earlier remarks, $(\mathcal{R}^{\mathrm{tt}}(K_1,\Sigma_1^*),\le ,r)$ is a topological Ramsey space.
\end{proof}

We now show that Theorem \ref{thm.2.10RS}  is recovered from Theorem \ref{thm.ex.1.10RamseySpace}.
Let $\bar{t}$, $l_k$ and $d_k$ be as in the assumption of Theorem \ref{thm.2.10RS}.
Since $\Rone$ is dense in $\PC_{\infty}^{\tt tt}(K_1,\Sigma_1^*)$, we may without loss of generality  assume $\bar{t}\in\Rone$.
Then note that
for each $k<\om$,
$\pos^{\tt tt}(\bar{t}\lre k)=\bigcup\{\val[s_{k}]:(s_0,\dots,s_k)\in r_{k+1}[k,\bar{t}\,]\}$.

To prove part (a) of Theorem \ref{thm.2.10RS}, for each $k<\om$, define a function $c_k:\mathcal{AR}_k|\bar{t}\ra l_k$ by
$c_k(x_0,\dots,x_k)=d_k(f)$
where $f$ denotes {\em the} member of $\val[\min(x_k)]$, $\min(x_k)$ being  defined as above.
By Claim \ref{claim.b}, there is an $\bar{s}\le \bar{t}$ in $\Rone$ with $m^{s_0}_{\dn}=m^{t_0}_{\dn}$ such that $c_0$ is constant on $r_1[0,\bar{s}]\cap\Sigma^*_1(\bar{s})$, and for each $k\ge 1$, $c_k$ is constant on $r_{k+1}[k,\bar{s}]$.
Let $\bar{w}\le\bar{s}$ be the member of $\Rone$ determined by 
 $w_0=s_0$, and
for $n\ge 1$,  $w_n$ is the member of $\Sigma^*_1(s_{2n-1},s_{2n})$ such that 
$i^{w_n}=i^{s_{2n}}$, $A^{w_n}$ is the set of the $n+1$-least members of $A^{s_{2n}}$.
(There is exactly one such member of $\Rone$ with these properties.)
Then $\bar{w}\le\bar{s}$  with $m^{w_0}_{\dn}=m^{t_0}_{\dn}$ and
for each  $i<\om$, for the $k$ such that $w_i\in\Sigma^*(\bar{t}\lre k)$,
 $d_k$ is constant on $\pos^{\tt tt}(\bar{w}\lre i)$.

Part (b) of Theorem \ref{thm.2.10RS} follows immediately from \bf A.4\rm:
Recalling that for each $s\in\mathcal{AR}_1$, $|\val[s]|=1$,
define   $c:\mathcal{AR}_1|\bar{t}\ra l$
by 
$c(s)=d_{k(s)}(f^s)$, where $f^s$ is {\em the} member of $\val[s]$ and 
$k(s)$ is  the integer such that $s\in\Sigma^*_1(\bar{t}\lre k(s))$.
By \bf A.4\rm,
there is an $\bar{s}\le\bar{t}$ in $\Rone$ such that $c$ is constant on $r_1[0,\bar{s}]$, 
and hence, $\bar{s}$ satisfies
 (b) of Theorem  \ref{thm.2.10RS}.

\begin{rem}
Although the sets $r_1[0,\bar{t}\,]$ and $\pos^{\tt tt}(\bar{t}\,)$ are very closely related, as shown above,
for any $s\in r_1[0,\bar{t}\,]$,
$\val[s]$ loses   the information of $i^s$ (and hence also $A^s$)  from $s$.
Thus, it does not seem  that Theorem \ref{thm.2.10RS}  would  imply  the pigeonhole principle for $\mathcal{AR}_1$ on the topological Ramsey space
$\Rone$, let alone the Abstract Ellentuck Theorem for the space, which follows from our Theorem \ref{thm.ex.1.10RamseySpace}.
\end{rem}

\bf \noindent Example 2.11 in \cite{Roslanowski/Shelah13}. \rm
Let $\mathbf{H}_2(n)=2$ for $n<\om$ and let $K_2$ consist of all FP creatures $t$ for $\bfH_2$ such that 
\begin{enumerate}
\item[$\bullet$]
$\emptyset\ne \dis[t]\sse[m^t_{\dn},m^t_{\up})$, 
\item[$\bullet$]
$\emptyset\ne \val[t]\sse{}^{\dis[t]}2$,
\item[$\bullet$]
$\nor[t]=\log_2(|\val[t]|)$.
\end{enumerate}
For $t_0,\dots,t_n\in K_2$ with $m^{t_l}_{\up}\le m^{t_{l+1}}_{\dn}$,
let $\Sigma_2(t_0,\dots,t_n)$ consist of all creatures $t\in K_2$ such that
$$
m^t_{\dn}=m^{t_0}_{\dn},\ 
m^t_{\up}=m^{t_n}_{\up},\
\dis[t]=\dis[t_{l^*}],
\mathrm{\ and\ }
\val[t]\sse\val[t_{l^*}], \mathrm{\ for\ some\ } l^*\le n.
$$
The partial ordering $\le$ on  $\PC_{\infty}(K_2,\Sigma_2)$ is defined as follows:
$\bar{t}\le \bar{s}$ if and only if there  is a
sequence $(u_n)_{n<\om}$
of finite subsets of $\om$ such that $\max(u_n)<\min(u_{n+1})$
and for each $n<\om$,
 $t_n\in\Sigma_2(\bar{s}\re u_n)$,
where $\bar{s}\re u_n$ denotes the sequence $(s_i:i\in u_n)$.
For $\bar{t}\in \PC_{\infty}(K_2,\Sigma_2)$, $\pos(\bar{t}\,)$ is defined to be $\bigcup\{\val[t_n]:n<\om\}$.

Ros{\l}anowski and Shelah proved that 
 $(K_2,\Sigma_2)$ is a loose FFCC pair for $\mathbf{H}_2$ which is simple except omitting and  has
 bigness.
Thus, the following Observation 2.8 (3) in 
 \cite{Roslanowski/Shelah13} applies to this
 example to yield the following.

\begin{prop}[Ros{\l}anowski/Shelah, \cite{Roslanowski/Shelah13}]\label{prop.RS2.11}
 For any  $\bar{t}\in \PC_{\infty}(K_2,\Sigma_2)$ and any coloring  $d:\pos(\bar{t}\,)\ra l$,
for some $l\ge 1$, 
there is an $\bar{s}\le\bar{t}$ such that
$c$ is constant on  $\pos(\bar{s})$.
\end{prop}

We point out that this is the same statement as Conclusion 3.10 in  \cite{Roslanowski/Shelah13}, the   only difference  being the hypotheses on the creating pair.

\begin{defn}[The space $(\mathcal{R}(\PC_{\infty}(K_2,\Sigma_2)),\le,r)$]
Let 
$$
\mathcal{R}(\PC_{\infty}(K_2,\Sigma_2))
=\{\bar{s}\in \PC_{\infty}(K_2,\Sigma_2): \forall l<\om,\ |\val[t_l]\, |=l+1\},
$$
with its inherited partial ordering.
Abbreviate this space by $\mathcal{R}(K_2,\Sigma_2)$.
\end{defn}

\begin{thm}\label{thm.ex.1.11RamseySpace}
$(\mathcal{R}(K_2,\Sigma_2),\le,r)$ is a topological Ramsey space which is dense in 
the  partial ordering of all  pure candidates  $\PC_{\infty}(K_2,\Sigma_2)$.
\end{thm}

\begin{proof}
It is clear that $\mathcal{R}(K_2,\Sigma_2)$ forms a dense subset of $\PC_{\infty}(K_2,\Sigma_2)$.
Towards proving  that   \bf A.4 \rm holds,
let $k\ge 1$ be fixed,  $\bar{t}\in\mathcal{R}(K_2,\Sigma_2)$, and 
 $c:r_k[k-1,\bar{t}\,]\ra 2$ be a given coloring.
Each $\bar{x}\in r_k[k-1,\bar{t}\,]$ is of the form $\bar{x}=(t_0,\dots,t_{k-2},x_{k-1})$,  with  $x_{k-1}\in \Sigma_2(\bar{t}\lre j)$ for some $j\ge k-1$,  and 
$\dis[x_{k-1}]=\dis[t_l]$
and $\val[x_{k-1}]\in[\val[t_l]]^k$,  for some 
$l\in[k-1,j]$.
For any $\bar{v}\in \mathcal{R}(K_2,\Sigma_2)$ and $j\ge k-1$,
let 
$$
X(\bar{v},j)=\{x_{k-1}:\bar{x}\in r_k[k-1,\bar{v}]\}\cap\Sigma_2(\bar{v}\lre j).
$$
Define a coloring $c'$  on the members of all such $X(\bar{v},j)$ by  
$$
c'(x_{k-1})=c(t_0,\dots,t_{k-2},x_{k-1}).
$$

The proof of \bf A.4 \rm proceeds  by a fusion argument as follows.
Letting $\bar{t}^{k-2}$ denote $\bar{t}$, 
for each $j\ge k-1$, given $\bar{t}^{j-1}$, Claim \ref{claim.c} below yields  a $\bar{t}^j\in [j,\bar{t}^{j-1}]$ such that 
$c'$ is constant on 
$X(\bar{t}^j,j)$.
Define $\bar{s}\le\bar{t}$ by letting 
 $r_{k-1}(\bar{s})=r_{k-1}(\bar{t}\,)$ and $s_j=t^j_j$ for  all $j\ge k-1$.
Then $\bar{s}$ has  the property that 
for each $j\ge k-1$, $c'$ is constant on $X(\bar{s},j)$.
Take a strictly increasing sequence $(j_i)_{i\ge k-1}$, (with $j_{k-1}\ge k-1$), such that $c'$ has the same value on  all $X(\bar{s}, j_i)$.
Define $\bar{w}\in [k-1,\bar{s}\,]$ by
$r_{k-1}(\bar{w})=r_{k-1}(\bar{t})$, and 
for $i\ge k-1$,
take $w_i$ to be any member of $\Sigma_2(s_{j_i},\dots,s_{j_{i+1}-1})$
such that $|\val[w_i]|=i+1$.
Then $\bar{w}\in [k-1,\bar{t}\,]$ and $c$ is constant on $r_k[k-1,\bar{w}\,]$.

\begin{claim}\label{claim.c}
For  $j\ge k-1$, given $\bar{t}^{j-1}$,  there is a $\bar{t}^j\in [j,\bar{t}^{j-1}]$ such that $c'$ is constant on $X(\bar{t}^j,j)$.
\end{claim}

\begin{proof}
Each $x_{k-1}\in  X(\bar{t}^{j-1},j)$
is completely determined by the triple 
$(n,l,\val[x_{k-1}])$, where
$x_{k-1}\in\Sigma_2(t_j,\dots,t_n)$ and $l\in[j,n]$ is such that $\val[x_{k-1}]\in[\val[t_l]]^k$.
Thus, we may regard $c'$ on $X(\bar{t}^{j-1},j)$  as a coloring of triples from
$$
\{ (n,l,J):j\le l\le n\ \mathrm{and\ } J\in[\val[t_l]]^k\}.
$$
Letting $K_l=\val[t_l]$, we see that 
$c'$ induces a coloring $c''$ on 
$$
\bigcup_{j\le l\le n}[K_l]^k
\times\prod \{K_i: j\le i\le n,\ i\ne l\}
$$
as follows:
For $j\le l\le n$, any $p_i\in K_i$ ($i\ne l$) and $J_l\in[K_l]^k$,  
define  
$$
c''(p_{k-1},\dots,p_{l-1},J_l,p_{l+1},\dots,p_n)=c'(n,l,J_l).
$$
By Theorem \ref{thm.newandimportant}, we obtain 
infinite sets $L=\{l_p:p\ge j\}$, $N=\{n_p:p\ge j\}$ such that
$j\le  l_{j}:=\min(L)\le n_{j}<l_{j+1}\le n_{j+1}<\dots$, 
 and  subsets
$H_i\sse K_i$ such that 
for each  $p\ge j$,  $|H_{l_p}|=p+1$, and for each $i\not\in  L$, $|H_i|=1$, 
and moreover, $c''$ is constant on 
$$
\bigcup_{n\in N}\bigcup_{l\in L\cap (n+1)}
[H_l]^k\times\prod \{H_i:    j\le i\le n, \ i\ne l\}.
$$
Take $\bar{t}^j\in [j,\bar{t}^{j-1}]$
such that 
for each $p\ge j$,  
$t^j_p$  is the creature 
in $\Sigma_2(t^{j-1}_{n_{p-1}+1},\dots,t^{j-1}_{n_p})$
determined by
$\dis[t^j_p]=\dis[t^{j-1}_{l_p}]$, and 
$\val[t^j_p]=  H_{l_p}$.
Then the coloring $c'$  is constant on $X(\bar{t}^j,j)$.
\end{proof}

Thus, Claim \ref{claim.c} holds, and by the fusion argument above, along with previous remarks about \bf A.1 \rm - \bf A.3 \rm holding, 
$\mathcal{R}(K_2,\Sigma_2)$ is a topological Ramsey space.
\end{proof}

 Proposition \ref{prop.RS2.11}  is recovered  immediately from Theorem \ref{thm.ex.1.11RamseySpace}:
Noting that  $\pos(\bar{t}\,)=\bigcup\{\val[t_n]:n<\om\}
=\bigcup\{\val[x]:x\in r_1[0,\bar{t}\,]\}$,
given any coloring $d:\pos(\bar{t}\,)\ra 2$,
define a coloring $c:r_1[0,\bar{t}\,]\ra 2$ by
$c(x)=d(f)$, where $\{f\}=\val[x]$.
By \bf A.4 \rm for $r_1[0,\bar{t}\,]$,
there is an $\bar{s}\le \bar{t}$  in $\mathcal{R}(K_2,\Sigma_2)$ such that 
$c$ is constant on $r_1[0,\bar{s}\,]$.
Hence, $d$ is constant on $\pos(\bar{s}\,)$.
\vskip.1in

\bf \noindent Example 2.13 in \cite{Roslanowski/Shelah13}. \rm
Let $N>0$ and $\mathbf{H}_N(n)=N$ for $n<\om$.
Let $K_N$ consist of all FP creatures $t$ for $\bfH_N$ such that 
\begin{enumerate}
\item[$\bullet$]
$\dis[t] =(X_t,\vp_t)$, where  $X_t\subsetneq[m^t_{\dn},m^t_{\up})$, 
and $\vp_t:X_t\ra N$,
\item[$\bullet$]
$\nor[t]=m^t_{\up}$,
\item[$\bullet$]
$\val[t]=\{f\in {}^{[m^t_{\dn},m^t_{\up})}N:\vp_t\sse f$ and $f$ is constant on $[m^t_{\dn},m^t_{\up})\setminus X_t\}$.
\end{enumerate}
For $t_0,\dots,t_n\in K_2$ with $m^{t_l}_{\up}= m^{t_{l+1}}_{\dn}$,
let $\Sigma_N(t_0,\dots,t_n)$ consist of all creatures $t\in K_N$ such that
\begin{enumerate}
\item[$\bullet$]
$m^t_{\dn}=m^{t_0}_{\dn}$,  
$m^t_{\up}=m^{t_n}_{\up}$,
$X_{t_0}\cup\dots\cup X_{t_n}\sse X_t$,
\item[$\bullet$]
for each $l\le n$, either $X_t\cap[m^{t_l}_{\dn},m^{t_l}_{\up})=X_{t_l}$ and $\vp_t\re[m^{t_l}_{\dn},m^{t_l}_{\up})=\vp_{t_l}$,\\
or $[m^{t_l}_{\dn},m^{t_l}_{\up})\subsetneq X_t$ and $\vp_t\re [m^{t_l}_{\dn},m^{t_l}_{\up})\in\val[t_l]$.
\end{enumerate}
The partial ordering $\le$ on  $\PC^{\tt tt}_{\infty}(K_N,\Sigma_N)$ is defined by
$\bar{t}\le \bar{s}$ if and only if there  is a
strictly increasing sequence $(j_n)_{n<\om}$ such that 
each $t_n\in\Sigma_N(s_{j_n},\dots,s_{j_{n+1}-1})$.

Ros{\l}anowski and Shelah proved that 
 $(K_N,\Sigma_N)$ is a tight FFCC pair for $\mathbf{H}_N$ which has the {\tt t}-multiadditivity and weak
 bigness, and is gluing [\cite{Roslanowski/ShelahBK}, 2.1.7] for each pure candidate in $\PC^{\tt tt}_{\infty}(K_N,\Sigma_N)$.
Thus,  {\it Conclusion 4.8} of \cite{Roslanowski/Shelah13} holds for this example; that is,  Theorem \ref{thm.2.10RS}, with  each instance of 
$ \PC^{\tt tt}_{\infty}(K_1,\Sigma_1)$
replaced with
 $ \PC^{\tt tt}_{\infty}(K_N,\Sigma_N)$, holds.

We show  that this forcing itself forms a topological Ramsey space.
The pigeonhole principle 
\bf A.4 \rm will follow from the Hales-Jewett Theorem in \cite{Hales/Jewett63}.
This space  is  extremely similar to the space of infinite sequences of  variable words, which Carlson showed to be a topological Ramsey space  in \cite{Carlson87}, and which corresponds to  the ``loose" version.
We point out that  {\it Conclusion 4.8} of \cite{Roslanowski/Shelah13}
for this example does not 
follow from
$\PC^{\tt tt}_{\infty}(K_N,\Sigma_N)$
 being a topological Ramsey space, since members $s\in r_1[0,\bar{t}\,]$ may have $\val[s]$ of any cardinality.

\begin{thm}\label{prop.2.13tRs}
 $(\PC^{\tt tt}_{\infty}(K_N,\Sigma_N),\le,r)$ is a topological Ramsey space.
\end{thm}

\begin{proof}
Let  $k\ge 2$ and $\bar{t}\in  \PC^{\tt tt}_{\infty}(K_N,\Sigma_N)$ be given.
There is a one-to-one correspondence $\sigma$ between 
$ r_k[k-1,\bar{t}\,]$ and the set of finite variable words on alphabet $N$:
For  $(t_0,\dots,t_{k-2},s)\in r_k[k-1,\bar{t}\,]$,
let $\sigma(s)$ denote 
the variable word 
$(l_{k-1},\dots,l_m)$ 
where 
$m\ge k-1$ is such that 
 $s
\in\Sigma_N(t_{k-1},\dots t_m)$
  and 
for each $i\in[k-1,m]$,  $l_i\in N$ if and only if $\vp_x\re [m^{t_i}_{\dn},m^{t_i}_{\up})\setminus X_{t_i}\equiv l_i$;
and  $l_i=v$  if and only if $X_s\cap [m^{t_i}_{\dn},m^{t_i}_{\up})=X_{t_i}$.

Given a coloring $c: r_k[k-1,\bar{t}\,]\ra 2$,  let $c'$ color the collection of all variable words on alphabet $N$
by $c'(\sigma(s))=c(t_0,\dots,t_{k-2},s)$.
By the Hales-Jewett Theorem, there is an infinite sequence of variable words $(x_i)_{i<\om}$ such that $c'$ is constant on all variable words of the form $x_{i_0}[\lambda_0]^{\frown}\dots^{\frown}x_{i_n}[\lambda_n]$, where each $\lambda_j\in N\cup\{v\}$ and at least one $\lambda_j=v$.
For each $i\ge k-1$, let $l(i)=|x_i|$, the length of the word $x_i$.
Let  $m_0=k-1+l(k-1)$, and 
given $i<\om$ and $m_i$,
 let $m_{i+1}=m_i+l(i)$.
Let $(l^i_0,\dots, l^i_{l(i)-1})$ denote $x_i$.
Define $s_{k-1}$ to be the member  of $\Sigma_N(t_{k-1},\dots,t_{m_0-1})$ such 
 that $\sigma(s_{k-1})=x_0$,
and in general, for $i\ge 1$ define $s_{k-1+i}$ to be the member of 
$\Sigma_N(t_{m_{i-1}},\dots,  t_{m_i-1})$
such that $\sigma(s_{k-1+i})=x_i$.
Letting $\bar{s}=r_{k-1}(\bar{t})^{\frown} (s_{k-1},s_k,\dots)$,
it is routine to check that $c$ is monochromatic on $r_k[k-1,\bar{s}]$.
Hence, \bf A.4 \rm holds for $k\ge 2$.

A fusion argument identical to the proof of Claim \ref{claim.b} yields the following.

\begin{claim}\label{claim.d}
Given $\bar{t}\in \PC^{\tt tt}_{\infty}(K_N,\Sigma_N)$ and colorings $c_k:
\mathcal{AR}_{k+1}|\bar{t}\ra l_k$,
 for some $l_k\ge 1$,
there is an $\bar{s}\le\bar{t}$ in $\PC^{\tt tt}_{\infty}(K_N,\Sigma_N)$ such that $m^{s_0}_{\dn}=m^{t_0}_{\dn}$,
$c_0$ is constant on $r_1[0,\bar{s}]\cap\Sigma_N(\bar{s})$, and for each $k\ge 1$,
the coloring $c_k$ on $r_{k+1}[k,\bar{s}]$ is constant.
\end{claim}

Finally, to prove \bf A.4 \rm for $\mathcal{AR}_1$,
 let $c:r_1[0,\bar{t}\,]\ra 2$ be given.
For each $k<\om$,
define a coloring $c_k:\mathcal{AR}_{k+1}|\bar{t}\ra 2$ by
$c_k(x_0,\dots, x_{k})=c(x_{k})$.
Take $\bar{s}\le \bar{t}$ satisfying Claim \ref{claim.d}.
There is a strictly increasing sequence $(k_j)_{j<\om}$
 such that
 the color of $c_{k_j}$ on
$ r_{k_j+1}[k_j,\bar{s}\,]$
 is the same
for all $j<\om$.
Take   $\bar{v}\le \bar{s}$
satisfying 
that  for each $j<\om$,
$m^{v_j}_{\dn}=m^{s_{k_j}}_{\dn}$ and $m^{v_j}_{\up}=m^{s_{k_{j+1}-1}}_{\up}$,
and $v_j\in\Sigma_N(s_{k_j},\dots,s_{k_{j+1}-1})$.
Then  $c$ is constant on $r_1[0,\bar{v}]$.

Thus, \bf A.4 \rm holds, and  hence the Theorem holds.
\end{proof}

\section{Remarks and Further Lines of Inquiry}

Whenever a forcing contains a topological Ramsey space as a dense subset,  this has implications for the properties of the  generic extension  and provides as well   Ramsey-theoretic techniques for streamlining proofs.
Although this note only 
showed that the pure candidates for  three examples of  creature forcings
contain dense subsets forming topological Ramsey spaces,
 the work here points to and lays some groundwork for several 
natural   lines of inquiry.

One obvious  line of exploration is to 
develop stronger versions and other variants of Theorem \ref{thm.newandimportant}
to obtain the pigeonhole principle for the pure candidates for other creating pairs, in particular  for Example 2.12 in \cite{Roslanowski/Shelah13}.  
Another is to develop this theory for the loose candidates, as we only considered tight types here.
A deeper  line of inquiry is  to determine  the implications  that the existence of a topological Ramsey space dense in a collection of pure candidates for a creating pair has for  the
forcing notion (with stems) generated by that creating pair.

The topological Ramsey spaces considered here force ultrafilters on base set $K$, a set of creatures, which  in turn generate ultrafilters on  a countable set of finite functions $\mathcal{F}_{\mathbf{H}}$.
The work here
yields   partition theorems of  Ros{\l}anowski and Shelah  in \cite{Roslanowski/Shelah13} for two of  examples
considered
 in Section \ref{sec.tRs}.
It will be interesting to see how their 
 Glazer methods 
interact with the product tree Ramsey methods more abstractly.
The hope is that this article has piqued the reader's interest to  investigate further the connections between creature forcings and topological Ramsey spaces, as such investigations  will likely 
will lead to new Ramsey-type  theorems and  new topological Ramsey spaces, while adding to the collection of available techniques and streamlining approaches to at least some genres of the myriad of  creature forcings.


\bibliographystyle{amsplain}
\bibliography{references}

\providecommand{\bysame}{\leavevmode\hbox to3em{\hrulefill}\thinspace}
\providecommand{\MR}{\relax\ifhmode\unskip\space\fi MR }
\providecommand{\MRhref}[2]{%
  \href{http://www.ams.org/mathscinet-getitem?mr=#1}{#2}
}
\providecommand{\href}[2]{#2}
\begin{thebibliography}{10}

\bibitem{Carlson87}
Timothy~J. Carlson, \emph{An infinitary version of the
  {G}raham-{L}eeb-{R}othschild theorem}, Journal of Combinatorial Theory Series
  A. \textbf{44} (1987), no.~1, 22--33.

\bibitem{Carlson/Simpson90}
Timothy~J. Carlson and Stephen~G. Simpson, \emph{Topological {R}amsey theory},
  Mathematics of {R}amsey theory{\rm, volume 5 of} {A}lgorithms and
  {C}ombinatorics, Springer, 1990, pp.~172--183.

\bibitem{DiPrisco/Llopis/Todorcevic04}
C.~A. DiPrisco, J.~Llopis, and S.~Todorcevic, \emph{Parametrized partitions of
  products of finite sets}, Combinatorica \textbf{24} (2004), no.~2, 209--232.

\bibitem{DiPrisco/Mijares/Nieto15}
Carlos DiPrisco, Jos{\'{e}}~Grigorio Mijares, and Jesus Nieto, \emph{Local
  {R}amsey theory. {A}n abstract approach}, arXiv:1506.03488 (2015).

\bibitem{DobrinenJSL15}
Natasha Dobrinen, \emph{High dimensional {E}llentuck spaces and initial chains
  in the {T}ukey structure of non-p-points}, Journal of Symbolic Logic
  \textbf{81} (2016), no.~1, 237--263.

\bibitem{DobrinenJML16}
\bysame, \emph{Infinite dimensional {E}llentuck spaces and
  {R}amsey-classification theorems}, Journal of Mathematical Logic \textbf{16}
  (2016), no.~1, 37 pp.

\bibitem{Dobrinen/Mijares/Trujillo14}
Natasha Dobrinen, Jos{\'{e}}~G. Mijares, and Timothy Trujillo,
  \emph{Topological {R}amsey spaces from {F}ra{\"{i}}ss{\'{e}} classes,
  {R}amsey-classification theorems, and initial structures in the {T}ukey types
  of p-points}, Archive for Mathematical Logic, special issue in honor of James
  E. Baumgartner, 34 pp, To Appear. (Invited submission).

\bibitem{Dobrinen/Todorcevic14}
Natasha Dobrinen and Stevo Todorcevic, \emph{A new class of
  {R}amsey-classification {T}heorems and their applications in the {T}ukey
  theory of ultrafilters, {P}art 1}, Transactions of the American Mathematical
  Society \textbf{366} (2014), no.~3, 1659--1684.

\bibitem{Dobrinen/Todorcevic15}
\bysame, \emph{A new class of {R}amsey-classification {T}heorems and their
  applications in the {T}ukey theory of ultrafilters, {P}art 2}, Transactions
  of the American Mathematical Society \textbf{367} (2015), no.~7, 4627--4659.

\bibitem{Hales/Jewett63}
A.~W. Hales and R.~I. Jewett, \emph{Regularity and positional games},
  Transactions of the American Mathematical Society \textbf{106} (1963),
  222--229.

\bibitem{Kondo38}
Motokiti Kond{\^{o}}, \emph{Sur l'uniformization des complementaires
  analytiques et les ensembles projectifs des la second classe}, Japan Journal
  of Mathematics \textbf{15} (1938), 197--230.

\bibitem{Pudlak/Rodl82}
Pavel Pudl{\'{a}}k and Vojtech R{\"{o}}dl, \emph{Partition theorems for systems
  of finite subsets of integers}, Discrete Mathematics \textbf{39} (1982),
  67--73.

\bibitem{Raghavan/Todorcevic12}
Dilip Raghavan and Stevo Todorcevic, \emph{Cofinal types of ultrafilters},
  Annals of Pure and Applied Logic \textbf{163} (2012), no.~3, 185--199.

\bibitem{Roslanowski/ShelahBK}
Andrzej Ros{\l}anowski and Saharon Shelah, \emph{Norms on possibilities. {I}.
  {F}orcing with trees and creatures}, vol. 141, American Mathematical Society,
  1999.

\bibitem{Roslanowski/Shelah13}
\bysame, \emph{Partition theorems from creatures and idempotent ultrafilters},
  Annals of Combinatorics \textbf{17} (2013), no.~2, 353--378.

\bibitem{TodorcevicBK10}
Stevo Todorcevic, \emph{Introduction to {R}amsey {S}paces}, Princeton
  University Press, 2010.

\bibitem{Todorcevic/Tyros13}
Stevo Todorcevic and Konstantinos Tyros, \emph{Subsets of products of finite
  sets of positive upper density}, Journal of Combinatorial Theory, Series A
  \textbf{120} (2013), 183--193.

\end{thebibliography}

\end{document}